\documentclass{amsart}[12pt]
\usepackage{amsmath, amsthm, amscd, amsfonts, amssymb, mathrsfs, eufrak, dsfont, bbm}

\makeatletter \oddsidemargin.9375in \evensidemargin \oddsidemargin
\newtheorem{theorem}{Theorem}[section]

\newtheorem{proposition}[theorem]{Proposition}
\newtheorem{corollary}[theorem]{Corollary}
\newtheorem{definition}[theorem]{Definition}
\newtheorem{example}[theorem]{Example}

\theoremstyle{remark}

\numberwithin{equation}{section}

\begin{document}

\noindent {\footnotesize }\\[1.00in]

\title[ On globally symmetric Finsler spaces]{On globally symmetric Finsler spaces}

\author[R. Chavosh Khatamy , R. Esmaili\\]{R. Chavosh Khatamy $^{*}$, R. Esmaili\\}

\address{Department of Mathematics,
Faculty of Sciences, Islamic Azad University, Tabriz Branch}
\email{\tt chavosh@tabrizu.ac.ir, r\_chavosh@iaut.ac.ir  }

\address{Department of Mathematics, Faculty of Sciences, Payame noor University, Ahar Branch}
\email{\tt Rogayyeesmaili@yahoo.com}


\keywords{Finsler Space, Locally symmetric Finsler space, Globally
Symmetric Finsler space,  Berwald space.
\\
\indent $^{*}$ The first author was supported by the funds of the
Islamic Azad University- Tabriz Branch, (IAUT) }

\begin{abstract}
The paper consider the symmetric of Finsler spaces. We give some
conditions about globally symmetric Finsler spaces. Then we prove
that these spaces can be written as a coset space of Lie group
with an invariant Finsler metric. Finally, we prove that such a
space must be Berwaldian.
\end{abstract}
\maketitle

\section{Introduction}
The study of Finsler spaces has important in physics and Biology
(\cite{5}), In particular there are several important books about
such spaces (see \cite{1}, \cite{8}). For example recently D.
Bao, C. Robels, Z. Shen used the Randers metric in Finsler on
Riemannian  manifolds (\cite{9} and \cite{8}, page 214). We must
also point out there was only little study about symmetry of such
spaces (\cite{3}, \cite{12}). For example E. Cartan has been
showed symmetry plays very important role in Riemannian geometry
(\cite{5} and \cite{12}, page 203).
%Now here, we first explain that
%general properties of a globally Symmetric Finsler space and
%finally, we prove that each such space is Berwaldian.
%\begin {definition}\label{def1.00}
%Let $V$ be a n-dimensional real vector space. A Minkowski norm on
%$V$ is a functional $F$ on $V$ which is smooth on $V-\{0\}$ and
%satisfies the following conditions.
%\begin{enumerate}
%  \item $F(u)\geq0, \forall u\in V;$
%  \item $F(\lambda u)=\lambda F(u),\forall\lambda>0;$
%  \item For any basis
%  $\varepsilon_{1},\varepsilon_{2},...,\varepsilon_{n}$ of $V$,
%  write $F(y)=F(y^{1},y^{2},...,y^{n})$ for $y=y^{j}\in
%  \varepsilon_{j}$. Then the Hessian matrix
%\begin{eqnarray*}
%(g_{ij}):=\left([\frac{1}{2}F^{2}]y^{i}y^{j}\right)
%\end{eqnarray*}
%is positive-definite at any point of $V-\{0\}$.
%\end{enumerate}
%\end {definition}
%\begin {definition}\label{def1.01}
%A smooth reversible Finsler metric is said to be Parallel if and
%only if $D\times R=0$.
%\end {definition}
\begin {definition}\label{def1.0}
A Finsler space is locally symmetric if, for any $p\in M$, the
geodesic reflection $s_{p}$ is a local isometry of the Finsler
metric.
\end {definition}
\begin {definition}\label{def1.1}
A reversible Finsler space $(M,F)$ is called globally Symmetric if
for any $p\in M$ the exists an involutive isometry  $\sigma_{x}$
$($that is, $\sigma_{x}^{2}=I$ but $\sigma_{x}\neq I)$ of such
that $x$ is an isolated fixed point of $\sigma_{x}$.
\end {definition}
\begin {definition}\label{def1.2}
Let $G$ be a Lie group and $K$ is a closed subgroup of $G$. Then
the coset space $G/K$ is called symmetric if there exists an
involutive automorphism $\sigma$ of $G$ such that
$$G^{0}_{\sigma}\subset K\subset G_{\sigma},$$
where $G_{\sigma}$ is the subgroup consisting of the fixed points
of $\sigma$ in $G$ and $G^{0}_{\sigma}$ denotes the identity
component of $G_{\sigma}$.
\end {definition}
\begin {theorem}\label{the1.3}
Let $G/K$ be a symmetric coset space. Then any G-invariant
reversible Finsler metric $($if exists$)$ $F$ on $G/K$ makes
$(G/K, F)$ a globally symmetric Finsler space $($\cite{8}, page
8$)$.
\end {theorem}
\begin {theorem}\label{the1.4}
Let $(M,F)$ be a globally Symmetric Finsler space. For $p\in M$,
denote the involutive isometry of $(M,F)$ at $p$ by $S_{x}$. Then
we
have \\

~~~~~~ $($a$)$ ~~  For any $p\in M, (dS_{x})_{x}=-I$. In
particular, $F$ must
  be reversible.

~~~~~~ $($b$)$ ~~ $(M,F)$ is forward and backward complete;

~~~~~~ $($c$)$ ~~ $(M,F)$ is homogeneous. This is, the group of
isometries of $(M,F), I(M,F)$, acts transitively on $M$.

~~~~~~ $($d$)$ ~~  Let $\widetilde{M}$ be the universal covering
space of $M$
  and $\pi$ be the projection mapping. Then
  $(\widetilde{M},\pi^{*}(F))$ is a globally Symmetric Finsler
  space, where $\pi^{*}(F)$ is define by
\begin{eqnarray*}
\pi^{*}(F)(q)=F((d\pi)_{\widetilde{p}}(q)), ~~ q\in
T_{\widetilde{p}} (\widetilde{M}),
\end{eqnarray*}
(See \cite{8} to prove).
\end {theorem}
\begin {corollary}\label{cor1.5}
Let $(M,F)$ be a globally Symmetric Finsler space. Then for any
$p\in M, s_{p}$ is a local geodesic Symmetry at $p$. The Symmetry
$s_{p}$, is unique. $($See prove of Theorem 1.2 and \cite{1}$)$
\end {corollary}
\section{A theorem on globally Symmetric Finsler spaces}
\begin {theorem}\label{the2.1}
Let $(M,F)$ be a globally Symmetric Finsler space. Then exits a
Riemannian Symmetric pair $(G,K)$ such that $M$ is diffeomorphic
to $G/K$ and $F$ is invariant under $G$.
\end {theorem}
\begin {proof}
The group $I(M,F)$ of isometries of $(M,F)$  acts transitively on
$M$ ($(C)$ of theorem 1.5). We proved that $I(M,F)$ is a Lie
transformation group of $M$ and for any $p\in M$ (\cite{12} and
\cite{7}, page 78), the isotropic subgroup $I_{p}(M,F)$ is a
compact subgroup of $I(M,F)$ (\cite{4}). Since $M$ is connected
(\cite{7}, \cite{10}) and the subgroup $K$ of $G$ which $p$ fixed
is a compact subgroup of $G$. Furthermore, $M$ is diffeomorphic
to $G/K$ under the mapping $gH\to g.p$ , $g\in G$ (\cite{7}
Theorem
2.5, \cite{10}). \\
As in the Riemannian case in page 209 of \cite{7}, we define a
mapping $s$ of $G$ into $G$ by $s(g)=s_{p}gs_{p}$, where $s_{p}$
donote the (unique) involutive isometry of $(M,f)$ with $p$ as an
isolated fixed point. Then it is easily seen that $s$ is an
involutive automorphism of $G$ and the group $K$ lies between the
closed subgroup $K_{s}$ of fixed points of $s$ and the identity
component of $K_{s}$ (See definition of the symmetric coset space,
 \cite{11}). Furthermore, the group $K$ contains no normal
subgroup of $G$ other than $\{e\}$. That is, $(G,K)$ is symmetric
pair. $(G,K)$ is a Riemannian symmetric pair, because $K$ is
compact.
\end {proof}
The following useful will be results in the proof of our aim of
this paper.
\begin {proposition}
Let $(M,\bar{F})$ be a Finsler space, $p\in M$ and $H_{p}$ be the
holonomy group of $\bar{F}$ at $p$. If $F_{p}$ is a $H_{p}$
invariant Minkowski norm on $T_{p}(M)$, then $F_{p}$ can be
extended to a Finsler metric $F$ on $M$ by parallel translations
of $\bar{F}$ such that $F$ is affinely equivalent to $\bar{F}$
$($\cite{5}, proposition 4.2.2$)$
\end {proposition}
\begin {proposition}
A Finsler metric $F$ on a manifold $M$ is a Berwald metric if and
only if it is affinely equivalent to a Riemannian metric $g$. In
this case, $F$ and $g$ have the same holonomy group at any point
$p\in M$ $($see proposition 4.3.3 of \cite{5}$)$.
\end {proposition}
Now the main aim
\begin {theorem}
Let $(M,F)$ be a globally symmetric Finsler space. Then $(M,F)$
is a Berwald space. Furthermore, the connection of $F$ coincides
with the Levi-civita connection of a Riemannian metric $g$ such
that $(M,g)$ is a Riemannian globally symmetric space.
\end {theorem}
\begin {proof}
We first prove $F$ is Beraldian. By Theorem 2.1, there exists a
Riemannian symmetric pair $(G,K)$ such that $M$ is diffeomorphic
to $G/K$ and $F$ is invariant under $G$. Fix a $G$- invariant
Riemannian metric $g$ on $G/K$. Without losing generality, we can
assume that $(G,K)$ is effective (see \cite{11} page 213). Since
being a Berwald space is a local property, we can assume further
that $G/K$ is simple connected. Then we have a decomposition
(page 244 of \cite{11}):
\begin{eqnarray*}
G/K=E\times G_{1}/K_{1}\times G_{2}/K_{2}\times ...\times
G_{n}/K_{n},
\end{eqnarray*}
where $E$ is a Euclidean space, $G_{i}/K_{i}$ are simply connected
irreducible Riemannian globally symmetric spaces, $i=1,2,...,n$.
Now we determine the holonomy groups of $g$ at the origin of
$G/K$. According to the de Rham decomposition theorem (\cite{2}),
it is equal to the product of the holonomy groups of $E$ and
$G_{i}/K_{i}$ at the origin. Now $E$ has trivial holonomy group.
For $G_{i}/K_{i}$, by the holonomy theorem of Ambrose and Singer
(\cite{12}, page 231, it shows, for any connection, how the
curvature form generats the holonomy group), we know that the lie
algebra $\eta_{i}$ of the holonomy group $H_{i}$ is spanned by
the linear mappings of the form
$\{\widetilde{\tau}^{-1}R_{0}(X,Y)\widetilde{\tau}\}$, where
$\tau$ denotes any piecewise smooth curve starting from $o$,
$\widetilde{\tau}$ denotes parallel displacements (with respect to
the restricted Riemannian metric) a long $\widetilde{\tau}$,
$\widetilde{\tau}^{-1}$ is the inverse of $\widetilde{\tau}$,
$R_{0}$ is the curvature tensor of $G_{i}/K_{i}$ of the restricted
Riemannian metric and $X,Y\in T_{0}(G_{i}/K_{i})$. Since
$G_{i}/K_{i}$ is a globally symmetric space, the curvature tensor
is invariant under parallel displacements (page 201 of
\cite{10},\cite{11}). So
\begin{eqnarray*}
\eta_{i}=span\{R_{0}(X,Y)|X,Y\in T_{0}(G_{i}/K_{i})\},
\end{eqnarray*}
(see page 243 of \cite{7}, \cite{11}).\\
On the other hand, Since $G_{i}$ is a semisimple group. We know
that the Lie algebra of $K^{*}_{i}=Ad(K_{i})\simeq K$ is also
equal to the span of $R_{0}(X,Y)$ (\cite{11}). The groups $H_{i}$,
$K^{*}_{i}$ are connected (because $G_{i}/K_{i}$ is simply
connected) (\cite{10} and \cite{11}). Hence we have
$H_{i}=K^{*}_{i}$. Consequently the holonomy group $H_{0}$ of
$G/K$ at the origin is
\begin{eqnarray*}
K^{*}_{1}\times K^{*}_{2}\times...\times K^{*}_{n}
\end{eqnarray*}
Now $F$ defines a Minkowski norm $F_{0}$ on $T_{0}(G/K)$ which is
invariant by $H_{0}$ (\cite{2}). By proposition 2.2, we can
construct a Finsler metric $\bar{F}$ on $G/K$ by parallel
translations of $g$. By proposition 2.3, $\bar{F}$ is Berwaldian.
Now for any point $p_{0}=aK\in G/K$, there exists a geodesic of
the Riemannian manifold $(G/K , g)$, say $\gamma(t)$ such that
$\gamma(0)=0, \gamma(1)=p_{0}$. Suppose the initial vector of
$\gamma$ is $X_{0}$ and take $X\in p$ such that $d\pi(X)=X_{0}$.
Then it is known that $\gamma(t)=\exp tX.p_{0}$ and $d\tau(\exp
tX)$ is the parallel translate of $(G/K, g)$ along $\gamma$
(\cite{11} and \cite{7}, page 208). Since $F$ is $G$- invariant,
it is invariant under this parallel translate. This means that $F$
and $\bar{F}$ concede at $T_{p_{0}}(G/K)$. Consequently they
concide everywhere. Thus $F$ is
a Berwald metric. \\
For the next assertion, we use a result of Szabo' (\cite{2}, page
278) which asserts that for any Berwald metric on $M$ there
exists a Riemannian metric with the same connection. We have
proved that $(M,F)$ is a Berwald space. Therefore there exists a
Riemannian metric $g_{1}$ on $M$ with the same connection as $F$.
In \cite{11}, we showed that the connection of a globally
symmetric Berwald space is affine symmetric. So $(M,F)$ is a
Riemannian globally symmetric space (\cite{7}, \cite{11}).
\end {proof}
 From the proof of theorem
2.4, we have the following corollary.
\begin {corollary}
Let $(G/K, F)$ be a globally symmetric Finsler space and
$g=\ell+p$ be the corresponding decomposition of the Lie
algebras. Let $\pi$ be the natural mapping of $G$ onto $G/K$.
Then $(d\pi)_{e}$ maps $p$ isomorphically onto the tangent space
of $G/K$ at $p_{0}=eK$. If $X\in p$, then the geodesic emanating
from $p_{0}$ with initial tangent vector $(d\pi)_{e}X$ is given by
\begin{eqnarray*}
\gamma_{d\pi.X}(t)=\exp tX.p_{0}.
\end{eqnarray*}
Furthermore, if $y\in T_{p_{0}}(G/K)$, then $(d\exp
tX)_{p_{0}}(Y)$ is the parallel of $Y$ along the geodesic (see
\cite{11}, \cite{7} proof of theorem 3.3).
\end {corollary}
\begin{example}
Let $G_{1}\big/ K_{1}$, $G_{2}\big/ K_{2}$ be two symmetric coset
spaces with $K_{1},K_{2}$ compact (in this coset, they are
Riemannian symmetric spaces) and $g_{1},g_{2}$ be invariant
Riemannian metric on $G_{1}\big/ K_{1}$, $G_{2}\big/ K_{2}$,
respectively. Let $M=G_{1}\big/ K_{1}\times G_{2}\big/ K_{2}$ and
$O_{1}, O_{2}$ be the origin of $G_{1}\big/ K_{1}, G_{2}\big/
K_{2}$, respectively and denote $O=(O_{1}, O_{2})$ (the origin of
$M$). Now for $y=y_{1}+y_{2}\in T_{O}(M)=T_{O_{1}}(G_{1}\big/
K_{1})+T_{O_{2}}(G_{2}\big/ K_{2})$, we define
\begin{eqnarray*}
F(y)=\sqrt{g_{1}(y_{1},y_{2})+g_{2}(y_{1},y_{2})+\sqrt[s]{g_{1}(y_{1},y_{2})^{s}+g_{2}(y_{1},y_{2})^{s}}},
\end{eqnarray*}
where $s$ is any integer $\geq2$. Then $F(y)$ is a Minkowski norm
on $T_{O}(M)$ which is invariant under $K_{1}\times K_{2}$
(\cite{4}). Hence it defines an $G$- invariant Finsler metric on
$M$ (\cite{6}, Corollary 1.2, of page 8246). By theorem 2.1,
(M,F) is a globally symmetric Finsler space. By theorem 2.4 and
(\cite{2}, page 266) $F$ is non-Riemannian.
\end{example}


\begin{thebibliography}{99}
\bibitem{1} D. Bao, C. Robles and Z. Shen, Zermelo navigation on
Riemannian manifolds, {\em J. Diff. Geom. 66 $($2004$)$, 377-435.}
\bibitem{2} D. Bao, S.S. Chern, Z. Shen. An Introduction to
Riemann-Finsler Geometry, {\em Springer- Verlag, New York, 2000.}
\bibitem{3} P. Foulon, Curvature and global rigidity in Finsler
manifolds, {\em Houston J. Math. 28.2 $($2002$)$, 263-292.}
\bibitem{4} P. Foulon, Locally symmetric Finsler spaces in negative
curvature, {\em C.R. Acad. Sci. Paris 324 $($1997$)$, 1127-1132.}
\bibitem{5} P. L. Antonelli, R.S. Ingardan and M. Matsumoto, The
Theory of Sprays and Finsler space with applications in Physics
and Biology, {\em Kluwer Academic Publishers, Dordrecht, 1993.}
\bibitem{6} S. Deng and Z. Hou, Invariant Finsler metrics on
homogeneous manifolds, {\em J. Phys. A: Math. Gen. 37 $($2004$)$,
8245-8253.}
\bibitem{7} S. Deng and Z. Hou, On locally and globally symmetric
Berwald space, {\em J. Phys. A: Math. Gen. 38 $($2005$)$,
1691-1697.}
\bibitem{8} S. Deng and Z. Hou, On symmetric Finsler space, {\em
IJM 216$($2007$)$, 197-219. }
\bibitem{9} S.S Chern, Z. Shen, Riemann-Finsler Geometry, {\em
WorldScientific, Singapore, 2004.}
\bibitem{10} S. Helgason, Differential Geometry, Lie groups and
Symmetric Spaces, {\em 2nd ed., Academic Press, 1978.}
\bibitem{11} S. Kobayashi, K. Nomizu, Foundations of Differential
Geometry, {\em Interscience Publishers, Vol. 1, 1963, Vol. 2,
1969.}
\bibitem{12}W. Ambrose and I. M. Singer, A theorem on holonomy, {\em
Trans. AMS. 75 (1953), 428-443.}
\end{thebibliography}
\end{document}